\documentclass[12pt]{amsart}
\usepackage{a4wide,enumerate,color,graphicx}
\usepackage{amsmath}
\allowdisplaybreaks

\let\pa\partial
\let\na\nabla
\let\eps\varepsilon

\newcommand{\R}{{\mathbb R}}
\newcommand{\C}{{\mathbb C}}
\newcommand{\diver}{\operatorname{div}}
\newcommand{\dom}{{\mathcal D}}

\newtheorem{theorem}{Theorem}
\newtheorem{lemma}[theorem]{Lemma}
\newtheorem{proposition}[theorem]{Proposition}
\newtheorem{remark}[theorem]{Remark}
\newtheorem{corollary}[theorem]{Corollary}

\begin{document}

\title[Cross-diffusion systems with entropy structure]{When do cross-diffusion systems \\
have an entropy structure?}

\author[X. Chen]{Xiuqing Chen}
\address{{School of Mathematics (Zhuhai), Sun Yat-sen University,
Zhuhai 519082, Guangdong Province, China}}
\email{buptxchen@yahoo.com}

\author[A. J\"ungel]{Ansgar J\"ungel}
\address{Institute for Analysis and Scientific Computing, Vienna University of
	Technology, Wiedner Hauptstra\ss e 8--10, 1040 Wien, Austria}
\email{juengel@tuwien.ac.at}

\date{\today}

\thanks{The first author acknowledges support from the National Natural Science
Foundation of China (NSFC), grant 11471050. The second author acknowledges
partial support from the Austrian Science Fund (FWF), grants P30000, W1245, and F65.}

\begin{abstract}
Necessary and sufficient conditions for the existence of an entropy structure
for certain classes of cross-diffusion systems with diffusion matrix $A(u)$
are derived, based on results from matrix factorization.
The entropy structure is important in the analysis for such equations since
$A(u)$ is typically neither symmetric nor positive definite.
In particular, the normal ellipticity of $A(u)$ for all $u$ and the symmetry of
the Onsager matrix implies its positive definiteness and hence an entropy structure.
If $A$ is constant or nearly constant in a certain sense,
the existence of an entropy structure is equivalent to the normal ellipticity of $A$.
Several applications and examples are presented, including the $n$-species
population model of Shigesada, Kawasaki, and Teramoto, a volume-filling model,
and a fluid mixture model with partial pressure gradients. Furthermore, the normal
elipticity of these models is investigated and some extensions are discussed.
\end{abstract}

\keywords{Cross diffusion, entropy method, normal ellipticity, matrix factorization,
Lyapunov equation, population model, volume-filling model, fluid mixture model.}

\subjclass[2000]{35K40, 35K55, 35Q92, 35Q79, 15A23, 15A24.}

\maketitle


\section{Introduction}

Cross-diffusion systems are systems of quasilinear parabolic equations
in which the gradient of one variable induces a flux of another variable.
They arise naturally in multicomponent systems from physics, chemistry, and
biology and describe, for instance, segregation in population species,
ion transport through nanopores, or dynamics of gas mixtures (see \cite{Jue16}).
A characteristic feature of most of these systems arising from applications
is that the diffusion matrix is generally neither symmetric nor positive
definite which significantly complicates the mathematical analysis.
However, it turns out that there might exist a transformation
of variables (called entropy variables) such that the transformed diffusion
matrix becomes positive definite and sometimes even symmetric. This is an
important ingredient in the global existence analysis of the equations.
The question is under which conditions does such a transformation exist?
In this paper, we will give some necessary and sufficient conditions for the
existence of entropy variables for certain classes of cross-diffusion systems.


\subsection*{The setting}

We consider the equations
\begin{equation}\label{1.eq}
  \pa_t u_i = \diver\bigg(\sum_{j=1}^n A_{ij}(u)\na u_j\bigg)
	\quad\mbox{in }\Omega,\ t>0,\ i=1,\ldots,n,
\end{equation}
subject to the initial and no-flux boundary conditions
\begin{equation}\label{1.bic}
  u_i(0) = u_i^0\quad\mbox{in }\Omega, \quad
	\sum_{j=1}^n A_{ij}(u)\na u_j\cdot\nu = 0\quad\mbox{on }\pa\Omega,\ t>0,\
	i=1,\ldots,n,
\end{equation}
where $\Omega\subset\R^d$ ($d\ge 1$) is a bounded domain,
$u_i:\Omega\times(0,\infty)\to\R$ are the unknowns (for instance, densities
or concentrations), $A_{ij}(u)\in\R$ are the diffusion coefficients,
and $\nu$ is the exterior unit normal vector to $\pa\Omega$.
More general equations, where $A_{ij}(u)$ are matrices, will be briefly
discussed in Section \ref{sec.conn}. We may add reaction terms to \eqref{1.eq},
but we concentrate in this paper on the diffusion operator.

Typically, $A(u)$ is neither symmetric nor positive definite (see the examples below)
such that even the local-in-time existence of solutions to \eqref{1.eq}--\eqref{1.bic}
is nontrivial. Amann \cite{Ama90} has shown that there exist local classical
solutions if the operator $\diver(A(u)\na(\cdot))$ is normally elliptic, which
means that all eigenvalues of $A(u)$ have positive real parts. By slightly
abusing the notation, we call such matrices normally elliptic.
This property is usually not sufficient for global-in-time existence.
In many applications, there exists a transformation
of variables $w=h'(u)$, where $h\in C^2(\dom)$ ($\dom\subset\R^n$ being a domain)
is called an entropy density and $h'$ its derivative,
such that \eqref{1.eq} can be written as
\begin{equation}\label{1.B}
  \pa_t u_i(w) = \diver\bigg(\sum_{j=1}^n B_{ij}(w)\na w_j\bigg)
	\quad\mbox{in }\Omega,\ t>0,\ i=1,\ldots,n,
\end{equation}
where $u(w):=(h')^{-1}(w)$, and $B(w)=A(u(w))h''(u(w))^{-1}$ is
positive definite in the sense that $B(w)+B(w)^\top$ is symmetric positive definite.
(We assume that the inverse functions exist.)
In this situation, $t\mapsto\int_\Omega h(u(t))dx$ is a Lyapunov functional
along the solutions to \eqref{1.eq}--\eqref{1.bic}.
Indeed, using $w=h'(u)$ as a test function
in \eqref{1.eq}, a formal computation yields
$$
  \frac{d}{dt}\int_\Omega h(u)dx + \int_\Omega \na u:h''(u)A(u)\na u dx = 0,
$$
where ``:'' denotes the Frobenius matrix product. Since $B(w)$ is assumed to be
positive definite, so does $h''(u)A(u)$, which shows that $t\mapsto
\int_\Omega h(u(t))dx$ is nonincreasing. Moreover, the second integral generally
provides gradient estimates which are essential in the existence analysis.
We say that \eqref{1.eq} possesses an entropy structure if there exists a
strictly convex function $h\in C^2(\dom)$ such that $B(w)$ or,
equivalently, $h''(u)A(u)$ is positive definite.

The aim of this paper is to explore under which conditions there exists an entropy
structure and how the corresponding entropy density can be constructed.
Furthermore, we will explore the connection between normal ellipticity and
the existence of an entropy structure.


\subsection*{Examples}

In the literature, an entropy structure has been found for specific
classes of cross-diffusion systems. For instance, segregating population species
can be modeled by equations \eqref{1.eq} with the diffusion coefficients
$$
  A_{ij}(u) = \delta_{ij}p_i(u)+ u_i\pa p_i/\pa u_j, \quad i,j=1,\ldots,n,
$$
where $p_i(u)$ are transition rates originating from the
lattice model from which these equations can be formally derived
\cite[Appendix A]{ZaJu17}.
This model was suggested by Shigesada, Kawasaki, and Teramoto \cite{SKT79}
for $p_i(u)=a_{i0}+\sum_{j=1}^n a_{ij}u_j$
and $n=2$. The $n$-species model with linear or nonlinear functions $p_i(u)$
was analyzed in \cite{CDJ18,DLMT15,Jue15,LeMo17}, and the existence of global weak
solutions was proved. Equations \eqref{1.eq} with diffusion coefficients
associated to linear functions $p_i$,
\begin{equation}\label{1.skt}
  A_{ij}(u) = \delta_{ij}\bigg(a_{i0} + \sum_{k=1}^n a_{ik}u_k\bigg) + a_{ij}u_i,
	\quad i,j=1,\ldots,n,
\end{equation}
where $a_{i0}\ge 0$, $a_{ij}\ge 0$, are called the SKT model.
It has an entropy structure
if there exist numbers $\pi_1,\ldots,\pi_n>0$ such that
$\pi_i a_{ij}=\pi_j a_{ji}$ for all $i\neq j$ holds. This assumption can be
recognized as the detailed-balance condition for the Markov chain
generated by $(a_{ij})$. Moreover,
the function $h(u)=\sum_{i=1}^n \pi_iu_i(\log u_i-1)$ is an entropy density,
and $h''(u)A(u)$ is positive definite for all $u\in\dom=\R_+^n$.

A second example are volume-filling models which describe multi-species systems
which take into account the finite size of the species and are given by
\eqref{1.eq} with the diffusion coefficients
\begin{equation}\label{1.vf}
  A_{ij}(u) = \delta_{ij} p_i(u)q_i(u_0) + u_ip_i(u)q'_i(u_0)
	+ u_iq_i(u_0)\pa p_i/\pa u_j,
\end{equation}
where $p_i$ and $q_i$ are transition rates (again, see \cite[Appendix A]{ZaJu17}
for a formal derivation)
and $u_0:=1-\sum_{i=1}^n u_i$ is the volume fraction of ``free space''
(in the context of biological models) or the solvent concentration (in the context
of fluid mixtures). The concentration vector $u=(u_1,\ldots,u_n)$ is an element
of $\dom=\{u\in\R_+^n:\sum_{i=1}^n u_i<1\}$, the so-called Gibbs simplex.
The existence of global weak solutions to system \eqref{1.eq} with $q_i=q$ in
\eqref{1.vf} was proved in \cite{ZaJu17}.
If there exists a convex function $\chi$ such that
$\pa\chi/\pa u_i=\log p_i$ for $i=1,\ldots,n$, then the function
$$
  h(u) = \sum_{i=1}^n u_i(\log u_i-1) + \int_a^{u_0}\log q(s)ds + \chi(u), \quad
	u\in\dom,
$$
is an entropy density, and $h''(u)A(u)$ is positive definite for $u\in\dom$.

A third example are equations for fluid mixtures driven by partial pressure
gradients,
\begin{equation}\label{1.pp}
  \pa_t u_i = \diver(u_i\na p_i(u)), \quad i=1,\ldots,n,
\end{equation}
where $u_i$ is the density of the $i$th fluid component
and $p_i$ is the $i$th partial pressure.
This model follows from the mass continuity equation $\pa_t u_i+\diver(u_iv_i)=0$
if the partial velocities $v_i$ are related to the partial pressures via Darcy's law,
$v_i=-\na p_i(u)$.
This system was derived from an interacting particle system in the mean-field
limit in \cite{CDJ19}. The entropy structure of this system is unknown up to now.
We determine conditions on the pressures $p_i$ under which \eqref{1.pp} has
an entropy structure.


\subsection*{Main results}

We sketch some of our main results. For details, we refer to the following sections.

\begin{itemize}
\item Section 2: If a matrix $A\in\R^{n\times n}$ can be decomposed as the
product $A_1A_2$ with a symmetric positive definite square matrix $A_1$,
the normal ellipticity or diagonalizability
of $A$ can be proved subject to properties of the square matrix $A_2$.
We collect known results
from matrix factorization theory and prove a new result characterizing the normal
ellipticity of $A$. We apply these findings in Sections \ref{sec.ent}--\ref{sec.popul}.

\item Section 3: Any cross-diffusion system with entropy
structure has a normally elliptic diffusion matrix. Thus, the normal ellipticity
is a necessary condition. Under this condition, if there exists a strictly convex
function $h\in C^2(\dom)$ ($\dom\subset\R^n$ being a domain)
such that $h''(u)A(u)$ is symmetric in $\dom$,
then $h''(u)A(u)$ is positive definite in $\dom$.
The symmetry requirement may be used to determine $h(u)$,
and we present some examples in this direction.

\item Section 4: If the diffusion matrix $A$ is constant, then its normal ellipticity
is equivalent to the existence of an entropy structure. Even if the normally
elliptic constant matrix $A$ is perturbed by a bounded nonlinear matrix, \eqref{1.eq}
has an entropy structure. Such a structure also exists if $h''(u)A(u)$
is symmetric up to a bounded nonlinear perturbation.

\item Section 5: If the entropy density is the sum of single-valued functions
and $h''(u)A(u)$ is symmetric, the positive definiteness of $h''(u)A(u)$ is equivalent
to the positivity of the leading principal minors of $A(u)$. This avoids the computation
of the eigenvalues of $A(u)$ to check its normal ellipticity. The idea allows
us to construct entropies in some situations, for instance for a general class
of $2\times 2$ diffusion matrices.

\item Section 6: If the matrix $(\pa p_i/\pa u_j)$ is normally
elliptic in $\dom$ and the detailed-balance condition
$$
  \pi_i \frac{\pa p_i}{\pa u_j}(u) = \pi_j\frac{\pa p_j}{\pa u_i}(u) \quad
	\mbox{for all }u\in\dom,\ i\neq j,
$$
holds, then the fluid mixture model \eqref{1.pp} has an entropy structure
with a Boltzmann-type entropy density.
Surprisingly, there exists a {\em second} entropy density, which is of quadratic type.
It is derived from the Poincar\'e lemma for closed differential forms by
interpreting the detailed-balance condition as the curl-freeness of the vector-field
$(\pi_1p_1,\ldots,\pi_np_n)$. These results are new.

\item Section 7: We prove that the diffusion matrix \eqref{1.skt} of the SKT model
is normally elliptic. Surprisingly, this property has not been
proved in the literature so far (except for the easy case $n=2$).
Furthermore, we investigate the normal ellipticity of the diffusion matrix
\eqref{1.vf} of the volume-filling model and \eqref{1.pp} of the fluid mixture model.
Also these results are new.
\end{itemize}

Finally, we discuss in Section \ref{sec.conn} some connections with results of
other authors and some extensions.


\subsection*{Definitions and notation}

We consider only real matrices $A\in\R^{n\times n}$ with coefficients
$A_{ij}$. The coefficients of a vector $u\in\R^n$ are denoted by $u_1,\ldots,u_n$.
The set $\sigma(A)$ signifies the spectrum of $A$. We say that the (possibly
nonsymmetric) matrix $A$ is positive definite if $z^\top Az>0$ for all
$z\in\R^n$, $z\neq 0$, or, equivalently, if $A+A^\top$
is positive definite. The matrix $A$ is normally elliptic if
$\operatorname{Re}(\lambda)>0$ for all $\lambda\in\sigma(A)$. In stability theory,
this property is sometimes called positive stability.
We say that $A$ is diagonalizable if there exists a
nonsingular matrix $P\in\R^{n\times n}$ such that $A=P\Lambda P^{-1}$ and
$\Lambda=\operatorname{diag}(\lambda_1,\ldots,\lambda_n)\in\R^{n\times n}$,
where $\lambda_i\in\sigma(A)$ for $i=1,\ldots,n$. In particular, the eigenvalues
of diagonalizable matrices are (here) real. We denote by $I\in\R^{n\times n}$ the
identity matrix.

Let $\dom\subset\R^n$ be a domain.
We say that \eqref{1.eq} has an entropy structure if there
exists a strictly convex function $h\in C^2(\dom)$ such that $h''(u)A(u)$
is positive definite for all $u\in\dom$. The matrix $A(u)h''(u)^{-1}$ is
called the Onsager matrix and it is symmetric and/or positive definite if and
only if $h''(u)A(u)$ is symmetric and/or positive definite, respectively.
The integral $\mathcal{H}(u)=\int_\Omega h(u)dx$ is called an entropy and
$-d\mathcal{H}/dt$ the entropy production.
Finally, we set $\R_+=(0,\infty)$ and $\C_+=\{y\in\C:\operatorname{Re}(y)>0\}$.


\section{Factorization of matrices}\label{sec.fac}

In this section, we collect some results concerned with the factorization of a
normally elliptic or diagonalizable matrix $A\in\R^{n\times n}$. Some of these
results are new. We will apply them to the diffusion matrix $A(u)$ from \eqref{1.eq}.
The factorization is based on the Lyapunov theorem for matrix equations;
see, e.g., \cite[Theorems 2.2.1 and 2.2.3]{HoJo91}.

\begin{theorem}[Lyapunov]
{\rm (i)} If $A\in\R^{n\times n}$ is normally elliptic then for any given
$G\in\R^{n\times n}$, there exists a unique matrix $H\in\R^{n\times n}$ such
that $HA+A^\top H=G$.

{\rm (ii)} The matrix $A\in\R^{n\times n}$ is normally elliptic if and only if
for a given symmetric positive definite matrix $G\in\R^{n\times n}$, there exists
a symmetric positive definite matrix $H\in\R^{n\times n}$ such that
$HA+A^\top H=G$.
\end{theorem}

We are analyzing factorizations $A=A_1A_2$ or $A=A_2A_1$ such that $A_1$
is symmetric positive definite. We will determine properties of $A$ when $A_2$ is
symmetric or positive definite and vice versa.

\begin{proposition}[Positive definite factorization]\label{prop.pdf}
The matrix $A\in\R^{n\times n}$ is normally elliptic if and only if there
exists a symmetric positive definite matrix $A_1$ and a positive definite matrix
$A_2$ such that $A=A_1A_2$ (or $A=A_2A_1$).
\end{proposition}

\begin{proof}
Let $A\in\R^{n\times n}$ be normally elliptic. By the Lyapunov theorem, there
exists a symmetric positive definite matrix $H$ such that $HA+A^\top H=I$.
Then $A_1=H^{-1}$ is symmetric positive definite and $A_2=HA$ satisfies
$A_2+A_2^\top=I$, i.e., $A_2$ is positive definite. This yields the desired
factorization $A=A_1A_2$. Furthermore, since $A^\top$ is normally elliptic,
the same argument shows that there exists a symmetric positive definite matrix
$A_1$ and a positive definite matrix $B$ such that $A^\top=A_1B$. We conclude that
$A=A_2A_1$ with $A_2:=B^\top$.

Assume that $A=A_1A_2$, where $A_1$ is symmetric positive definite and $A_2$
is positive definite. Set $H:=A_1^{-1}$. Then $A_2=HA$ and $HA+A^\top H = A_2+A_2^\top$
is positive definite. By the Lyapunov theorem, $A$ is normally elliptic.
If $A=A_2A_1$, the same argument can be applied to $A^\top=A_1A_2^\top$.
\end{proof}

The first part of the following result is proved in \cite[Theorem 6]{Bos87}.

\begin{proposition}[Symmetric factorization]\label{prop.sf}
{\rm (i)} The matrix $A\in\R^{n\times n}$ is diagonalizable if and only if 
there exists a
symmetric positive definite matrix $A_1$ and a symmetric matrix $A_2$ such that
$A=A_1A_2$ (or $A=A_2A_1$). 

{\rm (ii)} If $A=A_1A_2$ or $A=A_2A_1$
is normally elliptic with $A_1$ being symmetric positive definite and $A_2$
being symmetric, then $A_2$ is also positive definite.
\end{proposition}

\begin{proof}
It remains to prove part (ii). Indeed, we have $A_1^{-1}=A_2A^{-1}$
and hence $A_2A^{-1} + (A^{-1})^\top A_2 = A_1^{-1} + (A_1^{-1})^\top = 2A_1^{-1}$.
Since $A^{-1}$ is normally elliptic and $2A_1^{-1}$ is
symmetric positive definite, we conclude from the Lyapunov
theorem that there exists a unique symmetric positive definite matrix $H$ such that
$HA^{-1}+(A^{-1})^\top H=2A_1^{-1}$. The uniqueness of $H$ implies that $H=A_2$,
showing that $A_2$ is positive definite. The same argument can be made for $A=A_2A_1$.
\end{proof}

\begin{remark}[Eigenvalues of $A$]\rm
If $A=A_1A_2$ factorizes in a symmetric positive definite matrix $A_1$ and a symmetric
matrix $A_2$, Proposition \ref{prop.sf} implies in particular that the eigenvalues
of $A$ are real. We can say a bit more: By the inertia theorem of Sylvester
\cite[Section 1]{CaSc62}, the inertia of $A_1A_2$ and $A_2$ are the same, which
means that the number of positive, negative, and vanishing eigenvalues of $A$ and
$A_2$, respectively, are the same.
In particular, if $A_2$ has only positive eigenvalues, $A$ is normally elliptic. The
eigenvalues of $A$ can be bounded from below (or above) by the product of the
eigenvalues of $A_1$ and $A_2$; see, e.g., \cite[Theorem 2.2]{LuPe00} for details.
\qed
\end{remark}

\begin{remark}[Compatibility of factorizations]\label{rem.comp}\rm
Let $A=A_1A_2$ or $A=A_2A_1$ be a matrix factorization with a symmetric
positive definite matrix $A_1$. 
Proposition \ref{prop.sf} (ii) states that if $A$ is normally elliptic and $A_2$ 
is symmetric then $A_2$ is positive definite.
We may ask whether the diagonalizability of $A$
and positive definiteness of $A_2$ imply symmetry of $A_2$. The answer is no.
A counter-example is given as follows. Let
$$
  A = \begin{pmatrix} 1 & 2 \\ 0 & 2 \end{pmatrix}, \quad
	P = \begin{pmatrix} 1 & 2 \\ 0 & 1 \end{pmatrix}, \quad
	\Lambda = \begin{pmatrix} 1 & 0 \\ 0 & 2 \end{pmatrix}.
$$
Then $A=P\Lambda P^{-1}$ is diagonalizable and positive definite. The matrix
$A_1=I$ is symmetric positive definite, $A_2=A$ is positive definite, and
$A=A_1A_2=A_2A_1$. However, $A_2$ is not symmetric. Still, we can factorize
$A=A_1A_2$ with two symmetric positive definite matrices
$$
  A_1 = \begin{pmatrix} 5 & 2 \\ 2 & 1 \end{pmatrix}, \quad
	A_2 = \begin{pmatrix} 1 & -2 \\ -2 & 6 \end{pmatrix}.
$$
This motivates the following proposition.
\qed
\end{remark}

\begin{proposition}[Symmetric positive definite factorization]\label{prop.spdf}
The matrix $A\in\R^{n\times n}$ is normally elliptic and diagonalizable if
and only if it is a product of two symmetric positive definite matrices.
\end{proposition}

Note that $A\in\R^{n\times n}$ is normally elliptic and diagonalizable
if and only if $A$ is diagonalizable with positive eigenvalues. Thus,
the proposition is the same as \cite[Theorem 7]{Bos87}, but our proof is new.

\begin{proof}
The sufficiency follows from Propositions \ref{prop.pdf} and \ref{prop.sf}, while
the necessity is a consequence of Proposition \ref{prop.sf}.
\end{proof}

Table \ref{table} summarizes the factorization results.

\begin{table}[ht]
\centering
\caption{Types of factorization. PD = positive definite, S = symmetric,
NE = normally elliptic, D = diagonalizable.}
\begin{tabular}{|c|c|c|c|c|c|}
\hline
Factorization          & Prop.          & $A$ & $\sigma(A)$ & $A_1$ & $A_2$ \\ \hline
Positive definite      & \ref{prop.pdf} & NE  & $\C_+$      & S+PD  & PD \\
Symmetric              & \ref{prop.sf}  & D   & $\R$        & S+PD  & S \\
Symmetric positive definite & \ref{prop.spdf}& NE+D& $\R_+$      & S+PD & S+PD \\ \hline
\end{tabular}
\label{table}
\end{table}


\section{Necessary conditions for an entropy structure}\label{sec.ent}

We use the matrix factorization results to characterize the entropy structure
of \eqref{1.eq}.

\begin{theorem}\label{thm.main}
Let $A(u)\in\R^{n\times n}$ with $u\in\dom$.

{\rm (i)} If \eqref{1.eq} has an entropy structure then $A(u)$ is normally elliptic
for all $u\in\dom$.

{\rm (ii)} If $A(u)$ is normally elliptic for all $u\in\dom$ and there exists a
strictly convex function $h\in C^2(\dom)$ such that $h''(u)A(u)$ is symmetric for
all $u\in\dom$, then $h''(u)A(u)$ is positive definite for all $u\in\dom$, i.e.,
\eqref{1.eq} has an entropy structure.

{\rm (iii)} If \eqref{1.eq} has an entropy structure such that $h''(u)A(u)$ is
symmetric for all $u\in\dom$, then $A(u)$ is diagonalizable with positive eigenvalues.
\end{theorem}

\begin{proof}
The theorem follows from Propositions \ref{prop.pdf}, \ref{prop.sf} (ii), and
\ref{prop.spdf}.
We factorize $A(u)=A_1A_2$ with $A_1=h''(u)^{-1}$, which is symmetric positive definite,
and $A_2=h''(u)A(u)$.

(i) By assumption, $A_2$ is positive definite, so the result follows from
Proposition \ref{prop.pdf}. Another more elementary proof is given in
\cite[Lemma 3.2]{Jue17}.

(ii) As the matrix $A_2$ is assumed to be symmetric, Proposition \ref{prop.sf} (ii)
shows that $A_2$ is positive definite.

(iii) Proposition \ref{prop.spdf} implies that $A(u)$ is normally elliptic and
diagonalizable, which is equivalent to $A(u)$ being diagonalizable and having only
positive eigenvalues.
\end{proof}

\begin{remark}[Consequences]\label{rem.cons}\rm
The theorem can be used to determine whether an entropy structure exists.

(i) By Amann's result \cite[Section 1]{Ama90}, the normal ellipticity of $A(u)$
is a natural minimal condition for the local-in-time existence of smooth solutions.
If $A(u)$ is not normally elliptic, we cannot expect any entropy structure.

(ii) If $h''(u)A(u)$ is symmetric, so does the Onsager matrix $A(u)h''(u)^{-1}$.
The symmetry of the Onsager matrix is a natural condition imposed in general
systems consisting of irreversible thermodynamic processes.
If the application behind system \eqref{1.eq} should satisfy this principle,
we may calculate the entropy density by exploiting the symmetry of $h''(u)A(u)$;
see the examples below.

(iii) A simple check whether an entropy structure for \eqref{1.eq} exists with
a symmetric Onsager matrix is to compute the eigenvalues of $A(u)$. According to
Theorem \ref{thm.main}, if the diffusion matrix $A(u)$ is not diagonalizable with
positive eigenvalues, we cannot expect such a structure.
\qed
\end{remark}

We consider the following cases to detect an entropy structure.
Let $A(u)=A_1A_2$ or $A(u)=A_2A_1$, where $A_1$ is always symmetric positive definite.

\subsection*{Case 1.1.} Let $A(u)=A_1A_2$ and let $A_2$ be positive definite.
(According to Proposition \ref{prop.pdf}, $A(u)$ is
normally elliptic.) If we are able to find a function $h\in C^2(\dom)$
such that $h''(u)=A_1^{-1}$ (implying that $h$ is strictly convex),
then $A_2=h''(u)A(u)$ is positive definite, and
\eqref{1.eq} has an entropy structure.

As an example, we consider the Keller--Segel system with additional cross-diffusion:
\begin{equation}\label{2.ks}
  \pa_t u_1 = \diver(\na u_1-u_1 \na u_2), \quad
	\pa_t u_2 = \Delta u_1 + \delta\Delta u_2 + u_1 - u_2\quad\mbox{in }\Omega\subset\R^2,
\end{equation}
together with the initial and boundary conditions \eqref{1.bic}. The variables
$u_1$ and $u_2$ denote the cell density and the concentration of the chemical
signal, respectively. The parameter $\delta>0$ describes the strength of the
additional cross-diffusion. The classical parabolic-parabolic Keller--Segel model
is obtained when $\delta=0$. It is well known that this model has solutions
that blow up in finite time if $d\ge 2$ and the total mass is sufficiently
large \cite{CaCo08,CPZ04}.
System \eqref{2.ks} was suggested in \cite{HiJu11} to allow for global-in-time
solutions for any initial data.

The eigenvalues of $A(u)$ are $\lambda=1\pm \mathrm{i}\sqrt{\delta u_1}$, so
$A(u)$ is normally elliptic. We can factorize $A(u)=A_1A_2$ with
$$
  A_1 = \begin{pmatrix} u_1 & 0 \\ 0 & \delta \end{pmatrix}, \quad
	A_2 = \begin{pmatrix} 1/u_1 & -1 \\ 1 & 1/\delta \end{pmatrix},
$$
where $A_1$ is symmetric positive definite and $A_2$ is positive definite for $u_1>0$.
Then $h''(u)=A_1^{-1}$ and hence, $A_2=h''(u)A(u)$ is positive definite.
We can solve
$$
  h''(u) = \begin{pmatrix} 1/u_1 & 0 \\ 0 & 1/\delta \end{pmatrix}
$$
explicitly. Since $\pa^2 h/\pa u_1\pa u_2=0$, the entropy density is the
sum of $h_1(u_1)$ and $h_2(u_2)$ such that $h''_1(u_1)=1/u_1$, $h''_2(u_2)=1/\delta$.
This gives $h(u) = u_1(\log u_1-1)+u_2^2/(2\delta)$.

\subsection*{Case 1.2.} Let $A(u)=A_1A_2$ and let $A_2$ be symmetric.
(According to Proposition \ref{prop.sf}, $A(u)$ is
diagonalizable.) If we are able to find a function $h\in C^2(\dom)$
such that $h''(u)=A_1^{-1}$, then $A_2=h''(u)A(u)$ is symmetric. Thus, if
$A(u)$ is normally elliptic, Theorem \ref{thm.main} (ii) implies that \eqref{1.eq}
has an entropy structure, i.e., $h''(u)A(u)$ is positive definite.

To illustrate this result, we consider the $n$-species population model
\eqref{1.eq} with diffusion matrix \eqref{1.skt}.
The existence of global weak solutions was proved in \cite{CDJ18}
under the detailed-balance condition, i.e., there exist $\pi_1,\ldots,\pi_n>0$ such that
$$
  \pi_i a_{ij} = \pi_j a_{ji} \quad\mbox{for all }i\neq j.
$$
Under this condition, there exists a symmetric factorization $A(u)=A_1A_2$ with
$$
  (A_1)_{ij} = \frac{u_i}{\pi_i}\delta_{ij}, \quad
	(A_2)_{ij} = \frac{\pi_i}{u_i}\delta_{ij}\bigg(a_{i0} + \sum_{k=1}^n a_{ik}u_k\bigg)
	+ \pi_i a_{ij},
$$
where $i,j=1,\ldots,n$. Clearly, $A_1$ is symmetric positive definite if $u_i>0$,
while $A_2$ is symmetric. We set $h''(u)=A_1^{-1}$ and $A_2=h''(u)A(u)$.
We prove in Section \ref{sec.ne} that $A(u)$ is normally elliptic. Then
Theorem \ref{thm.main} (ii) shows that $h''(u)A(u)$ is positive definite, and
\eqref{1.eq} has an entropy structure. Clearly, the positive definiteness
of $h''(u)A(u)$ can be also verified directly; see \cite[Lemma 4]{CDJ18}.
We solve $h''(u)=A_1^{-1}$ by observing that $\pa^2 h/\pa u_i\pa u_j=0$ for
$i\neq j$ (so, $h(u)$ is the sum of some functions $h_i(u_i)$), and it follows that
$h''_i(u_i)=\pi_i/u_i$. We infer that $h(u)=\sum_{i=1}^n \pi_i u_i(\log u_i-1)$,
which is the entropy density suggested in \cite{CDJ18}.

\subsection*{Case 2.1.} Let $A(u)=A_2A_1$ and let $A_2$ be positive definite
(thus, $A(u)$ is normally elliptic). If $A_1=h''(u)$ and
$A_2=A(u)h''(u)^{-1}$ then \eqref{1.eq} has an entropy structure.

For instance, we wish to determine the entropy structure of the following system:
\begin{equation}\label{2.21}
  \pa_t u_1 = \frac12\Delta(u_1^2+u_3^2), \quad
	\pa_t u_2 = \frac12\Delta(u_1^2+u_2^2), \quad
	\pa_t u_3 = \frac12\Delta(u_2^2+u_3^2),
\end{equation}
together with the no-flux boundary conditions in \eqref{1.bic}.
This system is of the form
$\pa_t u = \Delta F(u)$, where $F:\dom\to\R^3$. The example was not considered
in the literature before. The diffusion matrix is given by
$$
  A(u) = \begin{pmatrix} u_1 & 0 & u_3 \\ u_1 & u_2 & 0 \\ 0 & u_2 & u_3 \end{pmatrix}.
$$
Then $A(u)=A_2A_1$, where
$$
  A_1 = \begin{pmatrix} u_1 & 0 & 0 \\ 0 & u_2 & 0 \\ 0 & 0 & u_3 \end{pmatrix}, \quad
	A_2 = \begin{pmatrix} 1 & 0 & 1 \\ 1 & 1 & 0 \\ 0 & 1 & 1 \end{pmatrix}.
$$
The matrix $A_2$ is positive definite and $A_1$ is symmetrix positive definite
if $u_i>0$ for $i=1,2,3$.
Then $h''(u)=A_1$, and $A_2=A(u)h''(u)^{-1}$ is positive definite, which provides the
entropy structure. The equation $h''(u)=A_1$ can be solved explicitly and
leads to the entropy density $h(u)=(u_1^3+u_2^3+u_3^3)/6$. Indeed, a formal
computation shows that along solutions to \eqref{2.21},
$$
  \frac{d}{dt}\int_\Omega h(u)dx + \frac12\sum_{i=1}^3\int_\Omega|\na u_i^2|^2 dx = 0.
$$

\subsection*{Case 2.2.} Let $A(u)=A_2A_1$ and $A_2$ be symmetric (then
$A(u)$ is diagonalizable).
If $h''(u)=A_1$, then $A_2=A(u)h''(u)^{-1}$ is symmetric. If $A(u)$ is also
normally elliptic, then $A(u)h''(u)^{-1}$ and consequently $h''(u)A(u)$
is positive definite, by Proposition \ref{prop.spdf}. We infer that \eqref{1.eq}
has an entropy structure.

As an example, consider the volume-filling model with diffusion matrix
$q_i=q$ in \eqref{1.vf}.  We assume that $q>0$, $q'>0$
and there exists a convex function $\chi$ such that $p_i=\exp(\pa\chi/\pa u_i)$
for $i=1,\ldots,n$.
Then $\pa p_i/\pa u_j=p_i\pa^2\chi/\pa u_i\pa u_j$ and consequently,
$$
  A_{ij}(u) = u_ip_i(u)q(u_0)\bigg(\frac{\delta_{ij}}{u_i} + \frac{q'(u_0)}{q(u_0)}
	+ \frac{\pa^2\chi}{\pa u_i\pa u_j}\bigg),
$$
recalling that $u_0=1-\sum_{i=1}^n u_i$. We can decompose $A(u)=A_2A_1$, where
$$
  (A_2)_{ij} = u_ip_i(u)q(u_0)\delta_{ij}, \quad
	(A_1)_{ij} = \frac{\delta_{ij}}{u_i} + \frac{q'(u_0)}{q(u_0)}
	+ \frac{\pa^2\chi}{\pa u_i\pa u_j}(u).
$$
Both $A_1$ and $A_2$ are symmetric positive definite for $u\in\dom$.
The entropy density can be computed from
$$
  \frac{\pa^2 h}{\pa u_i\pa u_j}(u) = \frac{\delta_{ij}}{u_i} + \frac{q'(u_0)}{q(u_0)}
	+ \frac{\pa^2\chi}{\pa u_i\pa u_j}(u)
$$
by integration, which leads, up to unimportant linear terms, to
$$
  h(u) = \sum_{i=1}^n u_i(\log u_i-1) + \int_a^{u_0}\log q(s)ds + \chi(u),
	\quad u\in\dom,
$$
where $a>0$. This is the same entropy density as used in \cite{ZaJu17}.


\section{Application: Perturbations}\label{sec.pdf}

We show some applications of Propositions \ref{prop.pdf} and \ref{prop.sf} (ii).
In particular, we
analyze perturbations of symmetric Onsager matrices and of constant diffusion matrices.

\begin{proposition}[Perturbation of $h''(u)A(u)$]\label{prop.pert}
Let $A(u)$ be normally elliptic uniformly in $\dom$ and diagonalizable.
Assume that there exists a strictly convex function $h\in C^2(\dom)$ such that
$h''(u)A(u) = S(u) + \eps N(u)$, where $S(u)$ is symmetric, $\eps>0$, and
$N$ is bounded in $\dom$. We also suppose that the eigenvalues of $S$ are bounded
in $\dom$ and the condition number
$\|A(u)\|\,\|A(u)^{-1}\|$ and $\|h''(u)^{-1}N(u)\|$ are bounded in $\dom$,
where the matrix norm is induced by the absolute norm in $\C^n$.
Then there exists $\eps_0>0$ such that
for all $0<\eps<\eps_0$, \eqref{1.eq} has an entropy structure.
\end{proposition}

\begin{proof}
Since $A(u)$ is diagonalizable, we can apply the Bauer--Fike theorem
\cite[Theorem 6.3.2]{HoJo13}: Let $\lambda(u)$ be an eigenvalue of $A(u)$
and $\mu(u)$ be an eigenvalue of $h''(u)^{-1}S(u)$. Then
$$
  \operatorname{Re}(\lambda(u))-\operatorname{Re}(\mu(u))
	\le |\lambda(u)-\mu(u)| \le  \eps\|A(u)\|\,\|A(u)^{-1}\|\,\|h''(u)^{-1}N(u)\|
	\le \eps C_1,
$$
where $C_1>0$ does not depend on $u\in\dom$. By assumption, there exists
$\lambda^*>0$ such that $\operatorname{Re}(\lambda(u))\ge\lambda^*$. Therefore,
we have $\operatorname{Re}(\mu(u))\ge\lambda^*/2$ for all
$0<\eps<\eps_1=\lambda^*/(2C_1)$. Thus, $h''(u)^{-1}S(u)$ is normally elliptic.
Moreover, we can decompose $h''(u)^{-1}S(u) = A_1A_2$, where $A_1=h''(u)^{-1}$
is symmetric positive definite and $A_2=S(u)$ is symmetric. We deduce from
Proposition \ref{prop.sf} (ii) that $S(u)$ is positive definite. By assumption on the
eigenvalues of $S$, there exists $\kappa>0$ such that for all
$z^\top S(u)z\ge\kappa|z|^2$. Thus, for all $z\in\R^n$,
$$
  z^\top h''(u)A(u)z = z^\top S(u)z + \eps z^\top N(u)z \ge (\kappa-\eps K)|z|^2,
$$
where $K=\|N(u)\|$. Thus, if $0<\eps<\eps_2<\kappa/K$, the matrix $h''(u)A(u)$
is positive definite, and the proof is finished after setting
$\eps_0=\min\{\eps_1,\eps_2\}>0$.
\end{proof}

\begin{remark}\rm
In cross-diffusion systems with volume filling, $h''(u)^{-1}$ may be uniformly bounded.
As an example, let $h(u)=\sum_{i=1}^3 u_i(\log u_i-1)$, where $u\in\dom=\{u\in\R_+^2:
\sum_{i=1}^2 u_i<1\}$ and $u_3=1-u_1-u_2$. Then
$$
  h''(u)^{-1} = \begin{pmatrix}
	u_1(u_2+u_3) & -u_1u_2 \\ -u_1u_2 & u_2(u_1+u_3) \end{pmatrix}
$$
is indeed bounded in $\dom$. Consequently, if $N(u)$ is bounded, so does
$h''(u)^{-1}N(u)$, which is one of the assumptions in Proposition \ref{prop.pert}.
\qed
\end{remark}

For a constant diffusion matrix, normal ellipticity and the existence of an
entropy structure are equivalent.

\begin{proposition}[Constant diffusion matrix]\label{prop.const}
If $A\in\R^{n\times n}$ is normally elliptic then \eqref{1.eq} has an entropy structure
and vice versa.
\end{proposition}

\begin{proof}
By Theorem \ref{thm.main} (i), an entropy structure implies that $A$ is normally
elliptic. Conversely, if $A$ is normally elliptic, by Proposition \ref{prop.pdf},
there exists a symmetric positive definite matrix $A_1$ and a positive definite
matrix $A_2$ such that $A=A_1A_2$. Defining the entropy density
$h(u)=\frac12 u^\top Hu$ with $H:=A_1^{-1}$, we infer that $h''(u)A=HA=A_2$
is positive definite.
\end{proof}

\begin{remark}[Explicit formula for $H$]\rm
The matrix $H$ appearing in the entropy density $h(u)=\frac12u^\top Hu$ can be
constructed explicitly. By the Lyapunov theorem, there exists a unique
symmetric positive definite matrix $H\in\R^{n\times n}$
such that $HA+A^\top H=I$. Then $HA+(HA)^\top=HA+A^\top H=I$ is symmetric
positive definite, i.e., $h''(u)A=HA$ is positive definite. According to
\cite[Problem 9, Section 2.2]{HoJo91}, it follows that
$$
  H = \int_0^\infty e^{-A^\top t}e^{-At} dt.
$$
An interesting consequence from this formula is that
$$
  \det H = \int_0^\infty\det(e^{-At})\det(e^{-At})dt
	= \int_0^\infty\det(e^{-2At})dt
	= \int_0^\infty e^{-2\operatorname{tr}(A)t}dt = \frac{1}{2\operatorname{tr}(A)},
$$
where we used the property $\det(e^{-2At})=e^{-2\operatorname{tr}(A)t}$
\cite[Theorem 2.12]{Hal15}.
\qed
\end{remark}

Proposition \ref{prop.const} can be slightly generalized to the sum of a
constant matrix and a nonlinear perturbation. Note that we do not assume that
$A(u)$ is diagonalizable.

\begin{proposition}[Perturbations of constant diffusion matrices]\label{prop.pertcon}
Let $A_0\in\R^{n\times n}$ be a constant normally elliptic matrix.

{\rm (i)} Let $A(u)=A_0+p(u)I$, where $p(u)$ is a positive scalar function.
Then \eqref{1.eq} has an entropy structure.

{\rm (ii)} Let $A(u)=A_0+\eps A_1(u)$, where $A_1(u)$ is a bounded matrix in $\dom$
and $\eps>0$.
Then there exists $\eps_0>0$ such that for all $0<\eps<\eps_0$, \eqref{1.eq}
has an entropy structure.
\end{proposition}

\begin{proof}
(i) By Proposition \ref{prop.const}, there exists a symmetric positive definite
matrix $H\in\R^{n\times n}$ such that $h''(u)A_0$ is positive definite, where
$h(u)=\frac{1}{2}u^\top Hu$. Then $h''(u)A(u)=HA_0+p(u)H$ is positive definite as the sum
of two positive definite matrices.

(ii) We know from part (i) that $HA_0$ is positive definite, where $H$ is symmetric
positive definite.
Thus, there exists $\lambda>0$ such that $z^\top HA_0z\ge\lambda|z|^2$ for all
$z\in\R^n$. Since $A_1(u)$ is bounded, there exists $M>0$ such that
$\|HA_1(u)\|\le M$ for all $u\in\dom$. We conclude that
$z^\top h''(u)A(u)z\ge (\lambda-\eps M)|z|^2$ for $z\in\R^n$,
and the positive definiteness follows after choosing $0<\eps<\eps_0<\lambda/M$.
\end{proof}


\section{Application: Sum of single-species entropy densities}\label{sec.spdf}

When the entropy density $h(u)$ can be written as the sum of functions
depending on $u_i$, we can give an easy criterion for the positive definiteness
of $h''(u)A(u)$, avoiding the computation of the eigenvalues of $A(u)$ in order
to check the normal ellipticity.

\begin{proposition}\label{prop.lpm}
If there exists a strictly convex function $h(u)=\sum_{i=1}^n h_i(u_i)$ for
some functions $h_i\in C^2(\dom)$ such that $h''(u)A(u)$ is symmetric, then
$h''(u)A(u)$ is positive definite if and only if all leading principal minors
of $A(u)$ are positive.
\end{proposition}

\begin{proof}
It holds that $h''(u)A(u)=(h_i''(u_i)A_{ij}(u))\in\R^{n\times n}$.
Let $M_k$ be the $k$th leading principal minor of $h''(u)A(u)$ and let
$A_k\in\R^{k\times k}$ be the leading principal submatrix of order $k$ of $A(u)$. Then
$$
  M_k = \prod_{i=1}^k h''_i(u_i)\det(A_k).
$$
Thus, if $h''(u)A(u)$ is symmetric positive definite, then
$M_k>0$ for all $k=1,\ldots,n$, by Sylvester's criterion
and hence $\det(A_k)>0$ for all $k=1,\ldots,n$. On the other hand,
if $\det(A_k)>0$ then $M_k>0$ for all $k=1,\ldots,n$, and we conclude from
the symmetry of $h''(u)A(u)$, by Sylvester's criterion again,
that $h''(u)A(u)$ is positive definite.
\end{proof}

The symmetry of $h''(u)A(u)$ implies that $h_i''(u_i)A_{ij}(u)=h_j''(u_j)A_{ji}(u)$
for all $i\neq j$. Since $h''_i(u_i)>0$, this shows that both $A_{ij}(u)$
and $A_{ji}(u)$ are positive, negative, or zero for any $i\neq j$.
We apply Proposition \ref{prop.lpm} to various examples.

\subsection*{Construction of entropies in two-species systems.}
We construct convex entropy densities $h(u)=h_1(u_1)+h_2(u_2)$ for \eqref{1.eq} with
the diffusion matrix
$$
  A(u) = \begin{pmatrix} a_{11}(u) & b_1(u_1)b_2(u_2) \\
	c_1(u_1)c_2(u_2) & a_{22}(u) \end{pmatrix},
$$
where $a_{11}>0$, $b_1b_2c_1c_2>0$, and $\det A>0$ in $\dom$. The symmetry of
$h''(u)A(u)$ is equivalent to $b_1b_2h_1''=c_1c_2h_2''$ or $b_1h_1''/c_1=c_2h_2''/b_2$.
The left-hand side depends only on $u_1$, the right-hand side only on $u_2$.
Thus, both sides are constant and, say, equal to $k\in\R$. Since $h_1''$ and
$h_2''$ are positive, we may set $k=\operatorname{sign}(b_1(u)c_1(u))|k|$.
(Note that the sign of $b_1(u)c_1(u)$ must be the same for all $u\in\dom$.) Then
$$
  h_1(u_1) = |k|\int_{u_1^*}^{u_1}\int_{v_1^*}^{v_1}\bigg|\frac{c_1(s)}{b_1(s)}\bigg|
	ds dv_1, \quad
	h_2(u_2) = |k|\int_{u_2^*}^{u_2}\int_{v_2^*}^{v_2}\bigg|\frac{b_2(s)}{c_2(s)}\bigg|
	ds dv_2,
$$
at least if these integrals exist. Without loss of generality, we may choose $|k|=1$.
Our assumptions imply that the leading principal minors of $A(u)$, namely
$a_{11}(u)$ and $\det(A(u))$, are positive. Thus, by Proposition \ref{prop.lpm},
$h''(u)A(u)$ is positive definite, and \eqref{1.eq} has an entropy structure.

\subsection*{Construction of entropies in $n$-species systems.}
The idea for two-species systems can be extended to $n\times n$ matrices.
To simplify, we consider entropy densities $h(u)=\sum_{i=1}^n h_i(u_i)$ and diffusion
matrices of the form
\begin{equation}\label{3.nn}
  A(u) = \begin{pmatrix}
	a_{11}(u) & a_{12}u_1 & a_{13}u_1 & \cdots & a_{1n}u_1 \\
	a_{21}u_2 & a_{22}(u) & a_{23}u_2 & \cdots & a_{2n}u_2 \\
	a_{31}u_3 & a_{32}u_3 & a_{33}(u) & & a_{3n}u_3 \\
	\vdots & \vdots & & \ddots & \vdots \\
	a_{n1}u_n & a_{n2}u_n & a_{n3}u_n & \cdots & a_{nn}(u) \end{pmatrix},
\end{equation}
where $a_{ij}\in\R$ and $u\in\R^n_+$.
We assume that both $a_{ij}$ and $a_{ji}$ are positive, negative,
or zero for any $i\neq j$ and that the leading principal minors of $A(u)$ are positive.
Matrices like \eqref{3.nn} appear, for instance, in diffusive population dynamics;
see \eqref{1.skt}. The matrix $h''(u)A(u)$ is symmetric if and only if
$$
  h_i''(u_i)u_i a_{ij}(u) = h_j''(u_j)u_j a_{ji}(u)\quad\mbox{for all }
	u\in\dom,\ i,j=1,\ldots,n.
$$
Hence, there exist constants $k_{ij}\in\R$ such that
$$
  h_i''(u_i)u_i a_{ij} = h_j''(u_j)u_j a_{ji} = k_{ij}
	= \operatorname{sign}(a_{ij})|k_{ij}|.
$$

{\em Case 1.} Let $a_{ij}a_{ji}>0$ for any $1\le i<j\le n$. Then
\begin{align*}
  h_i''(u_i) &= \frac{k_{ij}}{a_{ij}u_i} = \frac{|k_{ij}|}{|a_{ij}|u_i}
	\quad\mbox{for }i+1\le j\le n, \\
	h_j''(u_j) &= \frac{k_{ij}}{a_{ji}u_j} = \frac{|k_{ij}|}{|a_{ji}|u_j}
	\quad\mbox{for }1\le i\le j-1.
\end{align*}
We introduce the numbers $\pi_i=|k_{ij}|/|a_{ij}|$ for $i+1\le j\le n$ and
$\pi_j=|k_{ij}|/|a_{ji}|$ for $1\le i\le j-1$. Then it holds that $\pi_i>0$
for $i=1,\ldots,n$ and $\pi_i|a_{ij}| = \pi_j|a_{ji}|$ for $1\le i<j\le n$.
Since we assume that $a_{ij}a_{ji}>0$ for $i<j$, this yields the
detailed-balance condition
$$
  \pi_i a_{ij} = \pi_j a_{ji} \quad\mbox{for }i\neq j,
$$
which has been already imposed in \cite{CDJ18} (but we allow for negative
values of $a_{ij}$ and $a_{ji}$).
This shows that $h_i''(u_i)=\pi_i/u_i$ and hence, the entropy density
becomes
$$
  h(u) = \sum_{i=1}^n h_i(u_i) = \sum_{i=1}^n\pi_i u_i(\log u_i-1).
$$
Note that the detailed-balance condition is equivalent to the symmetry of
$h''(u)A(u)$.

{\em Case 2.} To simplify the presentation, we assume that there exists
only one couple of indices $(i_0,j_0)$ with $1\le i_0<j_0\le n$ such that
$a_{i_0j_0}=a_{j_0i_0}=0$. Case 1 applies to all indices $(i,j)\neq (i_0,j_0)$,
while for $(i,j)=(i_0,j_0)$, we have
\begin{align*}
  \pi_{i_0} &= \frac{|k_{i_0j}|}{|a_{i_0j}|}, \quad j=i_0+1,\ldots,j_0-1,j_0+1,
	\ldots,n, \\
  \pi_{j_0} &= \frac{|k_{ij_0}|}{|a_{ij_0}|}, \quad i=1, \ldots,i_0-1,i_0+1,\ldots,
	j_0-1.
\end{align*}
The detailed-balance condition $\pi_{i_0}a_{i_0j_0}=\pi_{j_0}a_{j_0i_0}$
is automatically satisfied since $a_{i_0j_0}=a_{j_0i_0}=0$.
For example, if $n=3$ and $a_{12}=a_{21}=0$,
the detailed-balance condition reduces to the identities
$\pi_1a_{13}=\pi_3a_{31}$ and $\pi_2a_{23}=\pi_3a_{32}$, while
$\pi_1a_{12}=\pi_2a_{21}$ is no longer needed.


\section{Application: fluid mixture and population models}\label{sec.popul}

We construct entropy densities for diffusive fluid mixture and population systems.

\subsection*{Fluid models with partial pressure gradients}

Let us consider the fluid mixture model \eqref{1.pp} which has the diffusion
matrix $A_{ij}(u)=u_i\pa p_i/\pa u_j$ for $i,j=1,\ldots,n$. Essential
for the analysis of is the following detailed-balance condition:
There exist numbers $\pi_1,\ldots,\pi_n>0$ such that
\begin{equation}\label{4.dbc}
  \pi_i\frac{\pa p_i}{\pa u_j} = \pi_j\frac{\pa p_j}{\pa u_i} \quad\mbox{in }\dom
	\mbox{ for all }i\neq j,
\end{equation}
and we assume that $\dom\subset\R_+^n$.

\begin{proposition}\label{prop.pp}
Assume that the matrix $Q=(\pa p_i/\pa u_j)$ is normally elliptic
and that the detailed-balance condition \eqref{4.dbc} holds. Then
\eqref{1.pp} has an entropy structure with the entropy density
$h(u)=\sum_{i=1}^n\pi_i u_i(\log u_i-1)$.
\end{proposition}

\begin{proof}
The normally elliptic matrix $Q$ factorizes according to $Q=A_1A_2$, where
$A_1=\operatorname{diag}(\pi_1^{-1},\ldots,\pi_n^{-1})$ is symmetric positive
definite and $A_2=(\pi_i\pa p_i/\pa u_j)$ is symmetric. Therefore, Proposition
\ref{prop.sf} (ii) shows that $h''(u)A(u)=A_2$ is positive definite.
\end{proof}

\begin{remark}[Alternative proof]\rm
We claim that the normal ellipticity of $A(u)$
is equivalent to that of $Q$. Then Proposition \ref{prop.pp} is
an immediate consequence of Theorem \ref{thm.main} (ii) since
$h''(u)A(u)=(\pi_i\pa p_i/\pa u_j)$ is assumed to be symmetric.
We use the inertia theorem of Sylvester
\cite[Section 1]{CaSc62}: If $A_1$ is symmetric positive definite and
$A_2$ is symmetric, then the inertia of $A_1A_2$ and $A_2$ coincide. As a consequence,
$A_1A_2$ and $A_2$ have the same number of eigenvalues with positive real parts.
We apply this result to $A_1=\operatorname{diag}(\pi_1^{-1},\ldots,\pi_n^{-1})$
and $A_2=(\pi_i\pa p_i/\pa u_j)$ to infer that $A_1A_2=Q$ and $A_2$ have the
same number of eigenvalues with positive real parts. The same argument applied
to $\widetilde A_1=\operatorname{diag}(u_1\pi_1^{-1},\ldots,u_n\pi_n^{-1})$ and
$A_2$ shows that also $\widetilde A_1A_2=A(u)$ and $A_2$ have the same number
of eigenvalues with positive part. This implies that the number of eigenvalues
with real parts of $A(u)$ and $Q$ coincide, proving the result.
\qed
\end{remark}


Interestingly, there exists a second entropy density.

\begin{proposition}[Second entropy]\label{prop.second}
Let $p_1,\ldots,p_n\in C^1(\dom)$ be defined on the simply connected set
$\dom\subset\R^n_+$ and assume that the detailed-balance condition \eqref{4.dbc} holds.
If $Q=(\pa p_i/\pa u_j)$ is invertible on $\dom$ then
\eqref{1.eq} has an entropy structure with an entropy density $h(u)$ satisfying
$\pa h/\pa u_i=\pi_ip_i$ for $i=1,\ldots,n$.
\end{proposition}

\begin{proof}
The proof is based on the Poincar\'e lemma for closed differential forms.
Since $(\pi_i p_i)$ defines a curl-free vector field in the sense
$\pa(\pi_ip_i)/\pa u_j=\pa(\pi_j p_j)/\pa u_i$ for all $i\neq j$, there
exists a function $h\in C^2(\dom)$ such that $\pa h/\pa u_i=\pi_ip_i$
for $i=1,\ldots,n$. It follows from the detailed-balance condition that
for all $z\in\R^n$,
\begin{align*}
  z^\top h''(u)A(u)z
	&= \sum_{i,j,k=1}^n\pi_i\frac{\pa p_i}{\pa u_k}u_k\frac{\pa p_k}{\pa u_j}z_iz_j
	= \sum_{i,j,k=1}^n\pi_ku_k\frac{\pa p_k}{\pa u_i}\frac{\pa p_k}{\pa u_j}z_iz_j \\
  &= \sum_{k=1}^n\pi_ku_k\bigg(\sum_{j=1}^n\frac{\pa p_k}{\pa u_j}z_j\bigg)^2\ge 0.
\end{align*}
Assume that $z^\top h''(u)A(u)z=0$ for $z\neq 0$. Since $u_k>0$, it follows that
$Qz=0$. However, $Q$ is invertible which implies that $z=0$, contradiction.
Thus, $z^\top h''(u)A(u)z>0$ for $z\neq 0$.
\end{proof}

Note that if the matrix $(\pa p_i/\pa u_j)$ is normally elliptic then it is
invertible. As an example, consider $p_i(u)=\sum_{j=1}^n a_{ij}u_j$
with coefficients $a_{ij}\ge 0$.
Then the Jacobian of $(p_1,\ldots,p_n)$ equals the matrix $(a_{ij})$. Thus,
if this matrix is invertible (and the detailed-balance condition holds),
Proposition \ref{prop.second} applies, and
\eqref{1.pp} with this choice has an entropy structure.
The entropy density can be constructed explicitly from
$\pa h/\pa u_i=\pi_i\sum_{j=1}^n a_{ij}u_j$ leading to
$$
  h(u) = \frac12\sum_{i,j=1}^n \pi_i a_{ij}u_iu_j.
$$
If $u$ is a (smooth) solution to \eqref{1.bic}, \eqref{1.pp} and
$\pa h/\pa u_i=\pi_i p_i$ for $i=1,\ldots,n$, it follows that
$$
  \frac{d}{dt}\int_\Omega h(u)dx
	+ \int_\Omega\sum_{i=1}^n\pi_i u_i|\na p_i(u)|^2 dx = 0,
$$
which yields slightly better integrability than the Boltzmann-type entropy density
from Proposition \ref{prop.pp}.


\subsection*{Population models}

We have considered the $n$-species SKT population model
with diffusion matrix \eqref{1.skt} already in Section \ref{sec.ent}.
Here, we study a more general version, given by
\begin{equation}\label{4.skt}
  \pa_t u_i = \Delta(u_i p_i(u)) = \diver(u_i\na p_i(u) + p_i(u)\na u_i), \quad
	i=1,\ldots,n,
\end{equation}
where $p_i(u)>0$ are transition rates from the underlying lattice model
\cite[Appendix A]{ZaJu17}. In the classical SKT model \eqref{1.skt},
we have $p_i(u)=a_{i0}+\sum_{j=1}^n a_{ij}u_j$.
Generally, the diffusion matrix has the elements $A_{ij}(u)=\delta_{ij}p_i(u)
+ u_i\pa p_i/\pa u_j$. Compared to the fluid model \eqref{1.pp}, system \eqref{4.skt}
contains the additional term $p_i(u)$ on the diagonal of the diffusion matrix.
Therefore, we have the same result as in Proposition \ref{prop.pp}.

\begin{corollary}[General SKT model]\label{coro.skt}
Let $p_1,\ldots,p_n:\dom\to\R_+$ be defined on the simply connected set
$\dom\subset\R^n_+$, let the matrix $Q=(\pa p_i/\pa u_j)$ be normally elliptic, and
let the detailed-balance condition \eqref{4.dbc} hold. Then
\eqref{4.skt} has an entropy structure with entropy density
$h(u)=\sum_{i=1}^n \pi_i u_i(\log u_i-1)$ for $u\in\dom$.
\end{corollary}

\begin{proof}
The matrix $h''(u)A(u)$ is the sum of the diagonal matrix with positive entries
$\pi_iu_i^{-1}p_i$ and the matrix $(\pi_i\pa p_i/\pa u_j)$. Since it follows
from Proposition \ref{prop.pp} that both matrices are symmetric
positive definite, so does $h''(u)A(u)$.
\end{proof}

This idea can be generalized to cross-diffusion systems of the form
\begin{equation}\label{4.F}
  \pa_t u_i=\Delta F_i(u), \quad i=1,\ldots,n.
\end{equation}
The diffusion matrix is given by the elements $A_{ij}(u)=\pa F_i/\pa u_j$
for $i,j=1,\ldots,n$. We recover \eqref{4.skt} for $F_i(u)=u_ip_i(u)$.

\begin{proposition}
Let $F_1,\ldots,F_n\in C^1(\dom)$ be defined on the simply connected set
$\dom\subset\R^n$, let the Jacobian of $(F_1,\ldots,F_n)$ be invertible and
let the detailed-balance condition \eqref{4.dbc} with $p_i$ replaced by $F_i$ hold.
Then \eqref{4.F} has an entropy structure.
\end{proposition}

\begin{proof}
By assumption, $(\pi_1F_1,\ldots,\pi_nF_n)$ is curl-free. Hence, we deduce
from the Poin\-car\'e lemma for closed differential forms the existence of a function
$h:\dom\to\R$ such that $\pa h/\pa u_i=\pi_iF_i$ for $i=1,\ldots,n$. This implies
for all $z\in\R^n$, $z\neq 0$ that
$$
  z^\top h''(u)A(u)z
	= \sum_{i,j,k=1}^n \pi_i\frac{\pa F_i}{\pa u_k}\frac{\pa F_k}{\pa u_j}z_iz_j
	= \sum_{i,j,k=1}^n \pi_k\frac{\pa F_k}{\pa u_i}\frac{\pa F_k}{\pa u_j}z_iz_j
  = \sum_{k=1}^n\pi_k\bigg(\sum_{j=1}^n\frac{\pa F_k}{\pa u_j}z_j\bigg)^2.
$$
Since $(\pa F_k/\pa u_j)$ is invertible, we infer as in the proof of Proposition
\ref{prop.second} that the previous expression is positive for all $z\neq 0$.
\end{proof}


\section{Normal ellipticity}\label{sec.ne}

We prove the normal ellipticity of the diffusion matrices associated to the
three models introduced in the introduction.

\subsection*{SKT model}

The normal ellipticity of the two-species SKT model
with diffusion matrix \eqref{1.skt} was proved in \cite[Section 17.1]{Ama93}.
Surprisingly, this property is not known for the $n$-species model. Under the
detailed-balance condition
$\pi_i a_{ij}=\pi_j a_{ji}$ for $i\neq j$, the existence of an entropy structure
was shown in \cite[Lemma 4]{CDJ18}, so that in this case $A(u)$ is normally
elliptic by Theorem \ref{thm.main} (i). In the following, we prove that
this property also holds when the detailed-balance condition is not valid.

\begin{lemma}\label{lem.sktne}
Define the matrix $A(u)\in\R^{n\times n}$ by \eqref{1.skt}, i.e.
$$
  A_{ii}(u) = a_{i0} + 2a_{ii}u_i + \sum_{j\neq i}a_{ij}u_j, \quad
	A_{ij}=a_{ij}u_i\quad\mbox{for }i\neq j.
$$
If $a_{i0}\ge 0$, $a_{ij}\ge 0$ with $a_{i0}+a_{ii}>0$ for $i,j=1,\ldots,n$,
then $A(u)$ is normally elliptic for any $u\in\R_+^n$.
\end{lemma}

\begin{proof}
We reformulate the matrix $A(u)$ by setting
$B_{ii}=a_{i0}+2a_{ii}u_i$ and $B_{ij}=a_{ij}u_j$ for $i\neq j$. Then
$A_{ii}=\sum_{j=1}^nB_{ij}$ and $A_{ij}=B_{ij}$ for $i\neq j$. We define
the matrix
$$
  \widetilde A = \begin{pmatrix}
	A_{11} & a_{12}u_2 & \cdots & a_{1n}u_n \\
	a_{21}u_1 & A_{22} & & a_{2n}u_n \\
	\vdots & & \ddots & \\
	a_{n1}u_1 & a_{n2}u_2 & & A_{nn} \end{pmatrix}
	= \begin{pmatrix}
	\sum_{j=1}^n B_{1j} & B_{12} & \cdots & B_{1n} \\
	B_{21} & \sum_{j=1}^n B_{2j} & & B_{2n} \\
	\vdots & & \ddots & \\
	B_{n1} & B_{n2} & & \sum_{j=1}^n B_{nn} \end{pmatrix}.
$$
Then $A(u)$ and $\widetilde A$ are similar since $\widetilde A=U^{-1}A(u)U$, where
$U=\operatorname{diag}(u_1,\ldots,u_n)$, and thus they have the same eigenvalues.
Since $B_{ii}=a_{i0}+2a_{ii}>0$ by assumption, the matrix $\widetilde A$
is strictly diagonally dominant. It follows from \cite[Theorem 6.1.10]{HoJo13}
that all eigenvalues of $\widetilde A$ have a positive real parts and so does
$A(u)$. This means that $A(u)$ is normally elliptic.
\end{proof}

This result can be generalized to population models of the form \eqref{4.skt},
where $A_{ij}(u)=p_i(u)\delta_{ij}+u_i\pa p_i/\pa u_j$. Indeed, if
$\pa p_i/\pa u_j\ge 0$ and $p_i(u)>\sum_{k\neq i}u_k\pa p_i/\pa u_k$ for
$u\in\dom$ and $i,j=1,\ldots,n$ then $A(u)$ is normally elliptic. For instance, if
$$
  p_i(u) = a_{i0} + \sum_{j=1}^n a_{ij}u_j^s, \quad i=1,\ldots,n,
$$
this condition is satisfied if $0<s<1$, $a_{i0}\ge 0$, $a_{ij}\ge 0$, and
$a_{i0}+a_{ii}>0$.

Another generalization concerns the condition $a_{i0}+a_{ii}>0$. It is not
necessary to conclude the normal ellipticity of $A(u)$. By applying the
Routh--Hurwitz stability criterion \cite[Section 2.3]{HoJo91} (or the
stability criterion of Li\'enard--Chipart \cite[Theorem 11, p.~221]{Gan59}),
we may allow for $a_{i0}+a_{ii}=0$. To avoid too many technicalities, we restrict
ourselves to the case $n=3$. Then the Routh--Hurwith criterion reads as follows:
The roots of the polynomial $\lambda^3+b_2\lambda^2+b_1\lambda+b_0$ have negative
real parts if and only if $b_i>0$ for $i=0,1,2$ and $b_2b_1>b_0$.

\begin{lemma}\label{lem.3}
Let $a_{i0}\ge 0$, $a_{ij}\ge 0$ for $i,j=1,2,3$ and set $b_{ij}=a_{ij}$ for
$i\neq j$ and $b_{ii}=a_{i0}+a_{ii}$. Assume that there exists a tripel
$(i,j,k)\in\{1,2,3\}^3$ such that
$$
  (i,j)\neq (2,1),\ (i,k)\neq (3,1),\ (j,k)\neq (3,2),
	\quad\mbox{and}\quad b_{1i}b_{2j}b_{3k} > 0.
$$
Then $A(u)$ is normally elliptic for any $u\in\R_+^3$.
\end{lemma}

\begin{proof}
The characteristic polynomial $p(\lambda)=\det(A(u)-\lambda I)$
equals $q(\lambda)=p(-\lambda)=\lambda^3 + b_2\lambda^2 + b_1\lambda + b_0$, where
$$
  b_0 = \det A, \quad
	b_1 = \sum_{1\le i<j\le 3}\det\begin{pmatrix} A_{ii} & A_{ij} \\ A_{ji} & A_{jj}
	\end{pmatrix}, \quad b_2 = \operatorname{trace}A.
$$
According to the Routh--Hurwitz criterion, we need to verify that $b_i>0$ for
$i=0,1,2$ and $b_2b_1-b_0>0$ to deduce the normal ellipticity of $A(u)$.
To this end, we recall the definition of $B_{ij}$ from the proof of
Lemma \ref{lem.sktne}: $B_{ii}=a_{i0}+2a_{ii}u_i$ and $B_{ij}=a_{ij}u_j$ for
$i\neq j$. Then $A_{ii}=\sum_{j=1}^3 B_{ij}$ and for $u\in\R_+^3$,
\begin{align*}
  b_2 &= \sum_{i,j=1}^3 B_{ij} > 0, \\
	b_1 &= \sum_{1\le i<j\le 3}(A_{ii}A_{jj}-B_{ij}B_{ji})
	= \sum_{1\le i<j\le 3}\sum_{(k,\ell)\neq(j,i)}B_{ik}B_{j\ell} > 0, \\
	b_0 &= A_{11}A_{22}A_{33} + B_{12}B_{23}B_{31} + B_{13}B_{32}B_{21} \\
	&\phantom{xx}{}- A_{11}B_{23}B_{32} - A_{22}B_{13}B_{31} - A_{33}B_{12}B_{21} \\
	&= \sum_{(i,j,k)}B_{1i}B_{2j}B_{3k} + B_{12}B_{23}B_{31}
	+ B_{13}B_{32}B_{21} > 0,
\end{align*}
where the sum is over all $(i,j,k)\in\{1,2,3\}^3$ such that
$(i,j)\neq(2,1)$, $(i,k)\neq(3,1)$, and $(j,k)\neq(3,2)$. Finally, we have
$$
  b_2b_1-b_0 \ge 2\sum_{i,j,k=1}^3 B_{1i}B_{2j}B_{3k}
	- B_{12}B_{23}B_{31} - B_{13}B_{32}B_{21}
	\ge \sum_{i,j,k=1}^3 B_{1i}B_{2j}B_{3k} > 0
$$
for $u\in\R_+^3$, finishing the proof.
\end{proof}

For instance, the matrix
$$
  A(u) = \begin{pmatrix}
	u_3 & 0 & u_1 \\
	u_2 & u_1 & 0 \\
	0 & u_3 & u_2 \end{pmatrix}
$$
satisfies the conditions of Lemma \ref{lem.3} with $(i,j,k)=(1,2,3)$.
Note that the detailed-balance condition is {\em not} satisfied for this matrix,
but Lemma \ref{lem.3} states that it is normally elliptic.
It is an open question whether \eqref{1.eq} with this diffusion matrix has
an entropy structure.


\subsection*{Volume-filling models}

We show the normal ellipticity of the diffusion matrix \eqref{1.vf} associated to
the volume-filling models in a special case.

\begin{lemma}\label{lem.vol}
Let $A(u)$ be defined by
$$
  A_{ij}(u) = \delta_{ij}p_i(u_i)q_i(u_0) + u_ip_i(u_i)q_i'(u_0)
	+ \delta_{ij}u_iq_i(u_0)\frac{\pa p_i}{\pa u_i}(u_i), \quad i,j=1,\ldots,n,
$$
for $i=1,\ldots,n$, $u\in\dom=\{u\in\R_+^n:\sum_{i=1}^n u_i<1\}$,
and we recall that $u_0=1-\sum_{i=1}^n u_i$. We assume that
$p_i$, $q_i$, and $q_i'$ are positive functions. Then $A(u)$ is normally elliptic
(and diagonalizable).
\end{lemma}

\begin{proof}
Since we can write
$$
  A_{ij}(u) = (\delta_{ij}u_ip_iq_i)\bigg(\delta_{ij}\bigg(\frac{1}{u_i}
	+ \frac{p_i'}{p_i}\bigg) + \frac{q_i'}{q_i}\bigg),
$$
we can decompose $A(u)=A_1A_2$, where
$$
  (A_1)_{ij} = \delta_{ij}u_ip_iq_i\bigg(\prod_{k\neq i}\frac{q_k'}{q_k}\bigg)^{-1},
	\quad
	(A_2)_{ij} = \delta_{ij}\bigg(\frac{1}{u_i}+\frac{p_i'}{p_i}\bigg)
	\prod_{k\neq i}\frac{q_k'}{q_k}+ \prod_{k=1}^n\frac{q_k'}{q_k}.
$$
Setting $V:=\prod_{k=1}^n q_k'/q_k$ and
$R_i:=(1/u_i+p_i'/p_i)\prod_{k\neq i}q_k'/q_k$, the matrix $A_2$ becomes
$$
  A_2 = \begin{pmatrix}
	R_1+V & V & \cdots & V \\
	V & R_2 + V & & V \\
	\vdots & & \ddots & \\
	V & V & & R_n+V \end{pmatrix}.
$$
This matrix is symmetric positive definite since the leading principle minors
are positive,
$$
  \prod_{k=1}^i R_k\bigg(1 + V\sum_{k=1}^i\frac{1}{R_k}\bigg) > 0, \quad
	i=1,\ldots,n
$$
(or since $z^\top A_2z=\sum_{i=1}^n R_iz_i^2 + (\sum_{i=1}^n z_i)^2V>0$ for $z\neq 0$).
We infer from Proposition \ref{prop.spdf} that $A(u)$ is normally elliptic and
diagonalizable.
\end{proof}


\subsection*{Fluid mixture model governed by partial pressure gradients}

The diffusion matrix $A(u)$ of the fluid mixture model \eqref{1.pp} is given by
$A_{ij}(u)=u_i\pa p_i/\pa u_j$. We consider the case of linear pressures,
$p_i(u)=\sum_{j=1}^n a_{ij}u_j$ with $a_{ij}\ge 0$.
Then $A_{ij}(u)=u_ia_{ij}$, and $A(u)$ can be decomposed
according to $A(u)=A_1A_2$ with $A_1=\operatorname{diag}(u_1,\ldots,u_n)$ and
$A_2=(a_{ij})$. The matrix $A_1$ is clearly symmetric positive definite.
Thus, $A(u)$ is normally elliptic if $(a_{ij})$ is positive definite
(but not necessarily symmetric).

This condition is sufficient but not necessary. Indeed, if $n=2$ and $a_{ij}\ge 0$,
$A(u)$ is normally elliptic in $\dom$ if and only if
$\det A_2=a_{11}a_{22}-a_{12}a_{21}>0$,
and in this case, the eigenvalues of $A(u)$ are positive.
On the other hand, $(a_{ij})$ is positive definite if and only if
$\det A_2>\frac14(a_{12}-a_{21})^2$, 
which is more restrictive than $\det A_2>0$ except if $A_2$ is symmetric.


\section{Connections and extensions}\label{sec.conn}

\subsection*{More general cross-diffusion systems}

Amann \cite[Section 4]{Ama93} considered a more general class of
cross-diffusion equations:
\begin{equation}\label{8.eq}
  \pa_t u_i = \sum_{j=1}^n\sum_{k,\ell=1}^d\frac{\pa}{\pa x_k}\bigg(A^{k\ell}_{ij}(u)
	\frac{\pa u_j}{\pa x_\ell}\bigg), \quad i=1,\ldots,n.
\end{equation}
These equations reduce to \eqref{1.eq} if $A^{k\ell}_{ij}(u)=\delta_{k\ell} A_{ij}(u)$.
Amann calls the differential operator on the right-hand side normally elliptic if
all the eigenvalues of its principal part $A_\pi(u,z)$ have positive real parts,
where
$$
  A_\pi(u,z) = \sum_{k,\ell=1}^d A^{k\ell}(u)z_kz_\ell, \quad u\in\dom,\ z\in\R^n.
$$
If $A^{k\ell}_{ij}(u)=\delta_{k\ell} A_{ij}(u)$, this coincides with our definition.
We say that \eqref{8.eq} has an entropy structure if there exists a strictly
convex function $h\in C^2(\dom)$ such that
\begin{equation}\label{8.pd}
  \sum_{k,\ell=1}^d h''(u)A^{k\ell}(u)z_kz_\ell
\end{equation}
is positive definite for all $u\in\dom$, $z\in\R^n$.
Our results can be extended in a straightforward way to this situation. For
instance, if $A_\pi(u,z)$ is normally elliptic and there exists a strictly convex
function $h\in C^2(\dom)$ such that $h''(u)A^{k\ell}(u)$ is
symmetric for all $u\in\dom$ and $k,\ell=1,\ldots,d$ then \eqref{8.eq} has an
entropy structure. This follows from the fact that $A_\pi(u,z)=A_1A_2$, where
$A_1=h''(u)^{-1}$ is symmetric positive definite and
$A_2=\sum_{k,\ell=1}^d h''(u) A^{k\ell}(u)z_kz_\ell$ is symmetric.
So, the claim follows from Proposition \ref{prop.sf} (ii).

\subsection*{Symmetrization by Kawashima and Shizuta}

The entropy structure of \eqref{8.eq} has been explored by Kawashima and
Shizuta \cite{KaSh88}, also including first-order terms. They call
$h\in C^2(\dom)$ an entropy for \eqref{8.eq} if $h$ is strictly convex,
the Onsager matrix $A^{k\ell}(u)h''(u)^{-1}$ is symmetric, and \eqref{8.pd} is
symmetric positive semidefinite
for all $u\in\dom$ and $k,\ell=1,\ldots,d$. In our situation, this is equivalent
to the symmetry and positive semidefiniteness of $h''(u)A(u)$.
It is shown in \cite{KaSh88} that if an entropy exists then system \eqref{8.eq}
can be symmetrized in the sense that it can be written as
\begin{equation}\label{8.w}
  \pa_t u_i(w) = \sum_{j,k,\ell=1}^n\frac{\pa}{\pa x_j}\bigg(B^{k\ell}_{ij}(w)
	\frac{\pa u_j}{\pa x_\ell}\bigg), \quad i=1,\ldots,n,
\end{equation}
where $B^{k\ell}(w) = A^{k\ell}(u(w))u'(w)$ is symmetric positive semidefinite
and $u(w):=(h')^{-1}(w)$. Conversely,
if there exists a diffeomorphism $w\mapsto u(w)$ such that the Jacobian
$(\pa w_i/\pa u_j)$ is symmetric and $B^{k\ell}(w)$ is symmetric positive semidefinite
then there exists an entropy. Indeed, by Poincar\'e's lemma for closed differential
forms and the symmetry $\pa w_i/\pa u_j=\pa w_j/\pa u_i$, there exists a function
$g$ such that $\pa g/\pa w_i=w_i$. Then $h(u)=u\cdot w(u) - g(w(u))$ is an
entropy for \eqref{8.eq}.

In the presence of first-order terms, we obtain hyperbolic-parabolic
balance laws:
$$
  \pa_t u_i + \sum_{j=1}^d \frac{\pa f_{ij}}{\pa x_j}(u)
	= \sum_{j=1}^n\sum_{k,\ell=1}^d\frac{\pa}{\pa x_k}\bigg(A^{k\ell}_{ij}(u)
	\frac{\pa u_j}{\pa x_\ell}\bigg),
$$
and the existence of an entropy requires the additional condition
$h'(u)f'_j(u)=q_j'(u)$ for some real-valued smooth functions $q_j$, where
$f_j=(f_{1j},\ldots,f_{nj})$ and $j=1,\ldots,d$. The first-order terms do not
modify the entropy production.
We refer to \cite[Theorem 2.1]{KaSh88} for details. Such systems have been also studied
in \cite{Ser10}, including rank deficient Onsager matrices. The extension of
our results to such situations is a future work. An example are Maxwell--Stefan
systems whose diffusion matrix has rank $n-1$; see \cite{Bot11,JuSt13}.
The symmetrizability property of the Euler equations was first observed by
Gudonov and later by Friedrichs and Lax and extended by Boillat to general
hyperbolic systems; see, e.g., \cite{BDLL96}.
Note that our definition of entropy does not need the symmetry of the Onsager matrix;
see Case 1.1 in Section \ref{sec.ent}.

\subsection*{Gradient-flow structure}

The entropy structure of \eqref{1.eq} is strongly related to the gradient-flow
formulation of Mielke and co-workers \cite{LiMi13,MPR16}. Let $M$ be a manifold,
$H:M\to\R$ be differentiable, and $K(u):T^*_uM\to T_uM$ be symmetric positive
definite for all $u\in M$, where $T_uM$ is the tangent space at $u\in M$
and $T^*_uM$ its cotangent space.
Physically, elements of $T_uM$ are the thermodynamic fluxes
and elements of $T^*_uM$ are the driving forces (here: entropy variables).
Then Mielke et al.\ call a solution $u:[0,T]\to M$ to
$$
  \pa_t u = -K(u)H'(u), \quad t>0,
$$
a gradient flow. The positive definiteness of $K(u)$ implies that
$$
  \frac{dH}{dt} = \langle H'(u),\pa_t u\rangle = -\langle H'(u),K(u)H'(u)\rangle \le 0,
$$
and we can interpret $H$ as an entropy. Here, $\langle\cdot,\cdot\rangle$ is
the dual paring between $T^*_uM$ and $T_uM$.
In the special case $K(u)=-\diver(B(w(u))\na(\cdot))$ and identifying
$H'(u)$ with $h'(u)$, we recover
\eqref{1.eq} since $K(u)H'(u)=\diver(B(w(u))\na h'(u))=\diver(A(u)\na u)$, where
$B(w(u))=A(u)h''(u)^{-1}$. Note that the mapping $\xi\mapsto K(u)\xi$ is linear.
This framework was generalized to nonlinear mappings $\xi\mapsto K(u,\xi)\xi$
in \cite{MPR16}. Physically, this means that the thermodynamics fluxes depend
nonlinearly on the driving forces. Then the gradient-flow equation reads as
$$
  \pa_t u = \pa_\xi\Psi(u,-H'(u)),\quad t>0,
$$
where $\Psi^*:TM\to\R$ is convex in its second argument and $\pa_\xi\Psi^*$ is the
partial derivative with respect to the second argument. This equation
is by Legendre--Fenchel theory equivalent to
$$
  \Psi(u,\pa_t u) + \Psi^*(u,-H'(u)) + \langle H'(u),u\rangle = 0.
$$
An example is $\Psi^*(u,\xi)=\frac12\langle\xi,K(u)\xi\rangle$.
The function $H(u)$ is an entropy in the sense
$$
  \frac{dH}{dt} = \langle H'(u),\pa_t u\rangle
	= -\Psi(u,\pa_t u) - \Psi^*(u,-H'(u)) \le 0,
$$
since $\Psi\ge 0$ and
$\Psi^*\ge 0$ (see \cite{MPR16} for details). It is an open question
to what extent the results of this paper can be extended to this case.


\end{document}